\documentclass[11pt]{amsart}
\usepackage[hmarginratio=1:1]{geometry}              
\usepackage[parfill]{parskip}    
\usepackage{graphicx}
\usepackage{amssymb}
\usepackage{epstopdf}
\usepackage{amsthm} 
\usepackage{marvosym}
\usepackage{enumerate}
\usepackage[all]{xy}
\usepackage[linktocpage]{hyperref}
\usepackage{float}

\usepackage{caption}
\usepackage{subcaption}

\newtheorem{thm}{Theorem}
\numberwithin{thm}{section}
\newtheorem{maintheorem}{Theorem}

\newtheorem{maincorollary}[maintheorem]{Corollary}

\newtheorem{lemma}[thm]{Lemma}

\newtheorem{proposition}[thm]{Proposition}

\newtheorem{conjecture}[thm]{Conjecture}

\newtheorem{corollary}[thm]{Corollary}

\theoremstyle{remark}
\newtheorem{rmk}[thm]{Remark}

\theoremstyle{definition}
\newtheorem{definition}[thm]{Definition}
\newtheorem{example}[thm]{Example}

\newcommand{\overl}[1]{\overline{#1}}

\newcommand{\R}{\mathbb{R}}

\newcommand{\Z}{\mathbb{Z}}

\newcommand{\T}{\mathbb{T}}
\newcommand{\E}{\mathbb{E}}
\newcommand{\Sa}{\mathbb{S}}
\newcommand{\Hy}{\mathbb{H}}

\newcommand{\Aut}{\text{Aut}}

\newcommand{\Isom}{\text{Isom}}

\newcommand{\lk}{\text{lk}}
\newcommand{\st}{\text{st}}
\newcommand{\Cone}{\text{Cone}}
\newcommand{\Int}{\text{int}}

\title{Isometry groups of CAT(0) cube complexes}
\author{Corey Bregman}
\begin{document}
\maketitle
\begin{abstract}Given a CAT(0) cube complex $X$, we show that if $\Aut(X)\neq \Isom(X)$ then there exists a full subcomplex of $X$ which decomposes as a product with $\R^n$. As applications, we prove that if $X$ is $\delta$-hyperbolic, cocompact and 1-ended, then $\Aut(X)=\Isom(X)$ unless $X$ is quasi-isometric to $\Hy^2$, and extend the rank-rigidity result of Caprace--Sageev to any lattice $\Gamma\leq \Isom(X)$.
\end{abstract}
\section{Introduction}\label{sec:Intro}
Let $X$ be a CAT(0) cube complex. Throughout, we will assume $X$ is finite-dimensional and locally finite. The purpose of this article is to investigate to what extent the CAT(0) metric on $X$ determines its cube complex structure.  By \textit{cubical automorphism} of $X$, we mean a bijective map $X\rightarrow X$ which takes cubes isometrically to cubes, preserving the combinatorial structure.  The CAT(0) metric on $X$ is the path metric induced by the Euclidean metric on each $n$-cube, hence any cubical automorphism of $X$ will be an isometry of the CAT(0) metric. Denote the automorphism group of $X$ by $\Aut(X)$ and the isometry group by $\Isom(X)$. By a lattice $\Gamma\leq \Isom(X)$, we will mean a discrete subgroup acting properly discontinuously and cocompactly on $X$.

The most basic example of a CAT(0) cube complex is $\R^n$, tiled in the usual way by cubes $[0,1]^n$. We will call this the \textit{standard cubulation} of $\R^n$.  The path metric space obtained from this cubulation is of course Euclidean $n$-space $\E^n$. The isometry group of Euclidean space $\Isom(\E^n)$ is isomorphic to  $\R^n\rtimes O(n)$ where $\R^n$ is the subgroup of translations and $O(n)$ is the orthogonal group. In contrast, the automorphism group of the standard cubulation is the group of isometries which preserve the orthogonal lattice $\Z^n\subset \E^n$, which can be described as $\Z^n\rtimes O(n,\Z)$.  Here, $O(n,\Z)$ is the signed symmetric group. Hence in the case of $\E^n$ the full isometry group is much larger than the automorphism group. Our main theorem states that this is essentially the only source of non-cubical isometries:
\begin{maintheorem} \label{main1}Let $X$ be a CAT(0) cube complex.  If $\Aut(X)\subsetneq\Isom(X)$, then there exists a full subcomplex of $X$ which decomposes as a cubulated product $Y\times \R^n$. If $n\neq 2$, the cube complex structure on $\R^n$ is standard.  
\end{maintheorem}

When $n=2$ above, it is possible that the Euclidean factor has a singular metric--the cube complex structure is built from squares with their edges paired, but there may be more than four squares appearing at a vertex. Such singular metrics arise naturally from locally CAT(0) cube complex structures on surfaces of genus $g\geq2$. In fact, if $X=\R^2$ is a CAT(0) cube complex with a singular metric and cocompact quotient, then $X$ is quasi-isometric to the hyperbolic plane $\Hy^2$. Since metrics on hyperbolic surfaces are generally more rigid than on tori, one might hope to show that $\Aut(X)=\Isom(X)$ when $X$ is cocompact.  In \S 6, we produce examples of singular cube complex structures on compact hyperbolic surfaces whose automorphism groups are strictly smaller than their isometry groups.  

On the other hand, as a corollary of the above theorem, we obtain a kind of rigidity for $\delta$-hyperbolic CAT(0) cube complexes.  

\begin{maintheorem} \label{main2}Let $X$ be $\delta$-hyperbolic, cocompact and 1-ended.  If $\Aut(X)\subsetneq \Isom(X)$, then $X$ is quasi-isometric to $\Hy^2$.
\end{maintheorem}

One reason to be interested in the full isometry group of a CAT(0) cube complex rather than just its automorphism group is the following conjecture of Ballmann--Buyalo \cite{BB08}:
\begin{conjecture}[Rank-rigidity] Let $X$ be a geodesically complete CAT(0) space and $\Gamma\leq \Isom(X)$ a lattice. If $X$ is irreducible, then $X$ is either a higher rank symmetric space, or a Euclidean building of dimension $n\geq 2$, or $\Gamma$ contains a rank-1 isometry. 
\end{conjecture}
A \textit{rank-1 isometry} is a hyperbolic isometry none of whose axes bounds a Euclidean half-plane. In \cite{CaSa11}, Caprace--Sageev, among many other results, proved the rank-rigidity conjecture in the case that $X$ is a CAT(0) cube complex and $\Gamma\leq \Aut(X)$:
\begin{thm}[Caprace--Sageev \cite{CaSa11}] Let $X$ be a geodesically complete CAT(0) cube complex and $\Gamma\leq\Aut(X)$ a lattice. If $X$ is irreducible, then $\Gamma$ contains a rank-1 isometry.
\end{thm}  
While $\Gamma\leq \Aut(X)$ is the main case of interest for CAT(0) cube complexes, it leaves open the question of rank rigidity when lattices in $\Isom(X)$ are not cubical or even virtually cubical.  For instance, many lattices in $\Isom(\E^n)$ are not virtually cubical, although they are always virtually isomorphic to $\Z^n$, which can be realized as a cubical lattice. Of course, this is not a counterexample, since $\E^n$ is not irreducible for $n>1$ and when $n=1$ every infinite isometry has rank 1. Our final result shows that the existence of Euclidean factors is the only way that a CAT(0) cube complex can have a non-cubical cocompact lattice.  
\begin{maintheorem}\label{main3}Let $X$ be a CAT(0) cube complex and $\Gamma\leq \Isom(X)$ a lattice.  Either there exists a finite index subgroup $\Gamma'\leq \Gamma$ such that $\Gamma'\leq \Aut(X)$, or $X=\E^n\times Y$ for some subcomplex $Y$.
\end{maintheorem}

Note that we do not assume geodesic completeness here.  As an immediate corollary, we extend the theorem of Caprace--Sageev to the full isometry group of a CAT(0) cube complex, settling the rank-rigidity conjecture for CAT(0) spaces isometric to a CAT(0) cube complex:
\begin{maincorollary}[Unrestricted Rank Rigidity] Let $X$ be an irreducible, geodesically complete CAT(0) space which is isometric to a CAT(0) cube complex and let $\Gamma\leq \Isom(X)$ be a lattice.  Then $\Gamma$ contains a rank-1 isometry.  
\end{maincorollary}

\textbf{Outline:} In \S 2, we review some background on CAT(0) cube complexes, and in \S 3 we introduce singular cube complexes on $\R^2$ and characterize those which are quasi-isometric to $\Hy^2$. In \S4 we develop the main technical tool, and then in \S 5 we prove Theorem \ref{main1} and deduce Theorems \ref{main2} and \ref{main3} as corollaries. Finally in \S6 we give a family of compact examples of non-positively curved cube complexes with isometry group different from their automorphism group.

\textbf{Acknowledgements:} The idea for this paper came from conversations with Anne Thomas and Jason Manning during a summer school in geometric group theory at MSRI in 2015. I would like to thank them both for getting me started on it.  I am also grateful to Jason for suggesting the way to get the higher-genus surface examples in \S6 by taking branched covers. I would also like to thank Michah Sageev and Fr\'ed\'eric Haglund for suggesting some possible approaches to Theorem \ref{main1}, and Andy Putman for several useful discussions and comments.


\section{Background on CAT(0) cube complexes}
Euclidean $n$-space will be denoted by $\E^n$.  We will often refer to $\R^n$ when we do not specify a metric structure, or refer to the vector space structure only. A Euclidean $n$-cube is the space isometric to $[0,1]^n\subset \E^n$.  A \textit{cube complex} is a metric space obtained by gluing together Euclidean $n$-cubes along their faces by isometries. The metric is the path metric induced by the Euclidean metric on each cube.

Let $X$ be a cube complex, and $v\in X^{(0)}$ a vertex.  The \textit{link} $\lk(v)$ is the boundary of a sufficiently small ball centered at $v$.  Since $X$ is a cube complex, $\lk(v)$ is naturally built out of simplices.  The link $\lk(v)$ is called \textit{flag} if every $(k+1)$-complete subgraph in the 1-skeleton $\lk(v)^{(1)}$ spans a $k$-simplex.  We will be interested in when the path metric on $X$ is CAT(0).  A theorem of Gromov tells us that this is completely determined by the links of vertices.

\begin{thm}[Gromov \cite{Gro87}] A simply connected cube complex $X$ is CAT(0) if and only if the link of each vertex of $X$ is a flag simplicial complex. 
\end{thm}

We will also need a general definition of a link of $n$-cube in a CAT(0) cube complex $X$. Let $C\subseteq X$ be a cube of dimension $n$.  Note that $C$ is isometrically embedded in $X$.  Suppose $C$ is contained in cubes $D_1,\ldots, D_k$.  The cubes $\{D_i\}$ define a poset by inclusion, and we can consider the simplicial realization $L'$ of this poset. As constructed, since each of the $D_i$ contribute vertices to $L'$, we have that $L'$ is actually the barycentric subdivision of a complex $L$, whose vertices are the $D_i$ whose dimension is $n+1$. Since $X$ is CAT(0), by Gromov's link condition the complex $L$ is actually flag simplicial.  
\begin{definition}The simplicial complex $L=\lk(C)$ defined above is called the \textit{(ascending) link of $C$}. A cube $C$ is called \textit{locally maximal} if $\lk(C)=\emptyset$.
\end{definition}

If $C=v$ is a vertex, clearly the ascending link defined as above is just the usual link $\lk(v)$.  Given a point $p\in C$, the link $\lk(p)$ naturally comes equipped with a simplicial structure isomorphic to $\Sigma_{n-1}*L$, where $*$ indicates a join and $\Sigma_{k}$ is the simplicial structure on the $k$-sphere $S^k$ given as the $k$-fold join of $S^0$, a disjoint union of two points. By convention $\Sigma_0$ is the empty set. In particular, we observe that if $p,q$ lie in the interior of the same cube $C$, then $\lk(p)\cong\lk(q)$.  

\subsection{Generalities on isometry groups of CAT(0) spaces} General results of Caprace--Monod \cite{CaMo09} state that the a CAT(0) space $X$ decomposes as a product $X=X_1\times \cdots\times X_n\times \E^k$, where the $X_i$ are irreducible in the sense that they do not decompose as a product, and none of them are Euclidean.  We get a corresponding decomposition of isometry groups as follows. Define $\Isom_0(X)$ to be the subgroup of $\Isom(X)$ which preserves each irreducible factor. Then $[\Isom(X):\Isom_0(X)]<\infty$ and \[\Isom_0(X)=\Isom(X_1)\times\cdots\times \Isom(X_n)\times\Isom(\E^k).\]In fact, $\Isom_0(X)$ is normal in $\Isom(X)$ and the quotient is just the finite group of permutation of isomorphic $X_i$ factors.  

If $X$ is a finite-dimensional CAT(0) cube complex, then there are only finitely many isomorphism types of cubes appearing in $X$ hence $X$ is an $M_\kappa$-polyhedral complex in the sense of \cite{BH99}, with $\kappa=0$. In this case, geodesic completeness of $X$ is equivalent to not having free faces (\cite{BH99}, Proposition 5.1). If $X$ is geodesically complete and $\Gamma\leq \Aut(X)$ is a cocompact lattice, then the normalizer $N(\Gamma)$ is a Lie group with finitely many connected components and $\Isom(X/\Gamma)\cong N(\Gamma)/\Gamma$.  Moreover, the identity component is a torus of rank equal to the rank of the center $Z(\Gamma)$ (\cite{BH99}, Theorem 6.17). Thus, if $\Gamma$ is centerless, $\Isom(X/\Gamma)$ is finite and therefore $\Aut(X/\Gamma)$ automatically has finite index in $\Isom(X/\Gamma)$.  If $\Gamma\leq \Isom(X)$ is any lattice, however, one does not know in general whether $\Gamma$ contains a finite-index subgroup of automorphisms, especially when $X$ is not geodesically complete.  

Regarding the relationship between lattices in the automorphism group of a CAT(0) cube complex and its geometry, much more is known.  Many general results such as the Tits alternative and the existence of non-trivial quasimorphisms were proven by Caprace--Sageev \cite{CaSa11}. Infinite lattices in $\Aut(X)$ are also known not to have Property (T) \cite{NR97}, since from a cocompact action of $\Gamma$ on $X$ one can build an action of $\Gamma$ on a $\R$-Hilbert space without an unbounded orbit. Some results concerning the relation between the automorphism group and the isometry group are known in the case of abelian subgroups.  Haglund \cite{Ha07} has proven that any proper action of $\Z$ on a CAT(0) cube complex $X$ preserves a combinatorial geodesic, \textit{i.e.} $\Z$ acts as translation along a preserved axis in the 1-skeleton.  More generally, Woodhouse \cite{W17} shows that a proper action of $\Z^n$ on $X$ preserves a cocompact, convex subcomplex quasi-isometric to $\R^n$.

\section{Singular cone metrics on $\R^2$} As alluded to in the introduction, in addition to $\E^n$ there is another basic example of a CAT(0) cube complex which admits non-cubical isometries.  Topologically, this space is homeomorphic to $\R^2$, but it is not isometric to $\E^2$. Before giving the example, we need to define the notion of a cone point. 
\begin{definition} A vertex $v$ in an NPC cube complex $X$ is said to be a \textit{cone point of order $n\geq 4$} if its link $\lk(v)$ is a polygon with $n$ sides. We call $v$ \textit{singular} if $n\geq 5$.
\end{definition}

\begin{example} $X=\Cone(\R^2,n)$.  We define $\Cone(\R^2,n)$ to be the 2-dimensional cube complex constructed as follows. Take $n$ quarter planes and glue them together in pairs cyclically, so that any two adjacent quarter planes share an infinite ray. The resulting cube complex, $\Cone(\R^2,n)$ is homeomorphic to $\R^2$ and if $n\geq 5$ has a single singular vertex, while at every other point the metric is locally Euclidean. See Figure \ref{fig:Cone} for an illustration.    
\begin{figure}[h]
\includegraphics[width=4in]{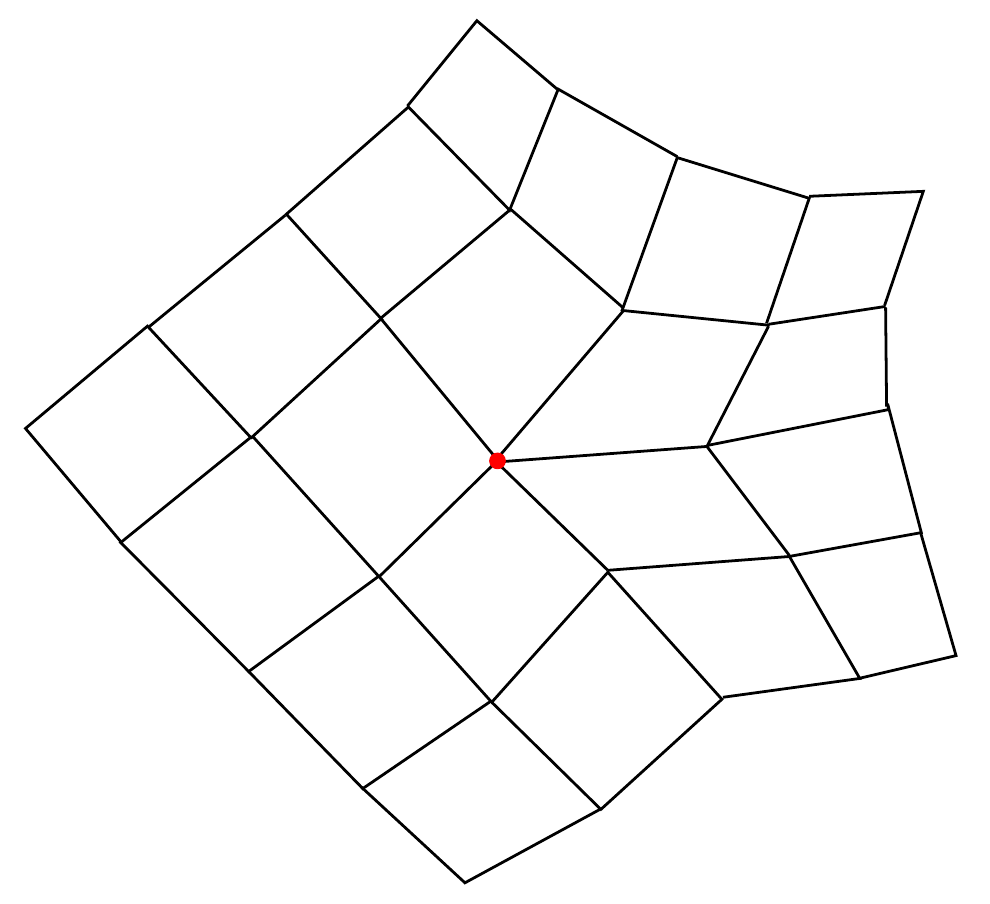}
\caption{A neighborhood of the cone point in $\Cone(\R^2,5)$.  All adjacent vertices have a locally Euclidean neighborhood.}
\label{fig:Cone}
\end{figure}

The isometry group of $X=\Cone(\R^2,n)$, must fix the unique singular vertex, but as every other point is locally Euclidean this is really the only restriction.  Therefore, $\Isom(X)\cong O(2)$. The cubical automorphism group must also preserve the singular vertex, and permute the squares incident to it. Hence, $\Aut(X)\cong D_{2n}$, the dihedral group of order $2n$. The link of the cone point in $\Cone(\R^2,n)$ is isometric to a circle with arc length $\frac{\pi n }{2}$. We denote this circle by $\Sa^1(n)$, so that the standard circle with arc length $2\pi$ is $\Sa^1=\Sa^1(4)$.     
\end{example}

More generally, a \textit{singular cone metric} on $\R^2$ is any CAT(0) cube complex built out of squares $[0,1]^2$ formed from a disjoint union of squares, where all edges are identified in pairs.  The CAT(0) requirement means that each vertex will be a cone of order at least 4. We require that there be at least one singular vertex.  As we will see, at least locally, a subcomplex isometric to either $\E^n$ or a singular cone metric on $\R^2$ is always present whenever a CAT(0) cube complex $X$ has an isometry which is not cubical. 

To end the section we give a characterization singular cone metrics on $\R^2$ which are quasi-isometric to the hyperbolic plane $\Hy^2$. It is likely that this characterization is known to the experts, but we prove it here in the special case of CAT(0) cube complexes, since we will need it in the sequel.  Let $X$ be a singular cone metric on $\R^2$.  First we characterize exactly when a singular cone metric is $\delta$-hyperbolic.  



\begin{lemma}\label{cobounded}Let $X$ be a singular cone metric on $\R^2$. $X$ is $\delta$-hyperbolic if and only if the set of cone points is cobounded in $X$.
\end{lemma}
\begin{proof} If the cone points are not cobounded in $X$, then for every integer $N>0$, there exists a point $x_N$ such that the ball of radius $N$ centered at $x_N$ is isometric to a ball of radius $N$ in $\E^2$.  It follows that $X$ is not $\delta$-hyperbolic for any $\delta$.  Conversely, observe that if $T\subset X$ is a geodesic triangle, then $X$ cannot contain any cone points in its interior.  Let $T$ be any geodesic triangle, and let $D$ be the constant such that every point of $X$ is within distance $D$ of a cone point. It follows from the observation that the circumcenter of $T$ is at most distance $D$ from each of the vertices.  In particular, $X$ satisfies the Rips condition with constant $2D$. 
\end{proof}

We will also need a lemma relating the order of cone points to a property called bounded growth at scale $R$, which we now define.
\begin{definition}[Bonk--Schramm \cite{BS00}]A metric space $X$ is said to have \textit{bounded growth} at scale $R$ if there exists $r<R$ and $N>0$ such that every ball of radius $R$ in $X$ can be covered by $N$ balls of radius $r$.  
\end{definition}

The next lemma states that for a singular Euclidean metric, bounded growth at some scale is equivalent to having bounded cone angles at all singular vertices. 
\begin{lemma} \label{AngleBound} Let $X$ be a singular cone metric on $\R^2$.  $X$ has bounded growth at some scale if and only if the cone angles of singular points are bounded.  
\end{lemma}
\begin{proof} First suppose the cone angles are bounded and $\leq N$.  Then the ball of radius 1 about any point can be covered by at most $N$ balls of radius $3/4$. Now assume that the cone angles are not bounded.  We claim that $X$ does not have bounded growth at any scale.  Let $p$ be a singular vertex of order $N$ and let $B_R(p)$ be the ball of radius $R$ center at $p$. Near $p$, $X$ is divided up into $N$ quarter planes, and for any point $x\in B_R(p)$, the geodesic between $x$ and $p$ lies in one of these quarter planes.  

Suppose for contradiction that $X$ has bounded growth at scale $R$, so that there exists $r<R$ and $M>0$ so that any ball of radius $R$ can be covered by at most $M$ balls of radius $r$.  Note that $M$ is at least 2.  Since the cone angles are unbounded, there exists some point $p$ with cone angle $N>4M$. Let $x_1,\cdots, x_m$ be the centers of the balls of radius $r$ which cover $B_R(p)$. Each of the $x_i$ lies in at most 2 quarter planes, so by the pigeonhole principle, there are at least $N-2M>2M\geq 4$ quarter planes which do not contain one of the $x_i$. For any point $z$ in one of these quarter planes, the geodesic between $z$ and $x_i$ passes through $p$.  It follows that since $r<R$, there are points not covered by the balls centered at the $x_i$, a contradiction.      
\end{proof}

The visual boundary at any point $x_0\in X$ is the space of equivalence classes of geodesic rays based at $x_0$, where two rays are equivalent if they stay at bounded distance from one another.  If $X$ is $\delta$-hyperbolic, the visual boundaries at any two points of $X$ are canonically homeomorphic.  For a singular cone metric on $\R^2$ the visual boundary is homeomorphic to $S^1$.

%

\begin{proposition}\label{ClassifySingularCone} Let $X$ be a singular cone metric on $\R^2$.  The following are equivalent:
\begin{enumerate}
\item $X$ is quasi-isometric to $\Hy^2$.
\item The singular cone points have bounded angles and are cobounded in $X$.  
\end{enumerate}
\end{proposition}
\begin{proof}
The 1-skeleton $X^{(1)}$ is quasi-isometric $X$; hence is hyperbolic by Lemma \ref{cobounded} and has boundary a topological circle by the remarks preceding the proposition. If $X$ has bounded angles, then $X^{(1)}$ has bounded vertex degree. By Theorem 11.3 and Corollary 11.6 of \cite{BS00}, the 1-skeleton $X^{(1)}$ will therefore be quasi-isometric to $\Hy^2$ if the union of all geodesics is cobounded in $X^{(1)}$.  It is not hard to see that every edge of $X^{(1)}$ lies on at least one bi-infinite geodesic. For the other direction, if $X$ is quasi-isometric to $\Hy^2$, then $X$ is hyperbolic, hence the cone points of $X$ are cobounded.  Moreover, by Theorem 11.2 of \cite{BS00}, $X$ has bounded growth at some scale, which is equivalent to having bounded angles by Lemma \ref{AngleBound}.
\end{proof}

\section{The local structure of an isometry}\label{sec:trace}
\noindent
In this section we study isometries of CAT(0) cube complexes locally, and see how the local picture of an isometry which is not cubical determines the global geometry. The main idea is that if an isometry is not cubical, then some points in lower-dimensional cubes will map into the interiors of cubes of higher dimensions. This forces the metric to be locally Euclidean at these points. From this observation, we are able to develop an isometry locally to find large convex subcomplexes which decompose as products with $\R^n$.
\subsection{The trace of an isometry.}
Let $X$ be a CAT(0) cube complex.  If $C\subset X$ is a closed $n$-cube, then $C$ is isometrically embedded in $X$.  By a \textit{chart} we mean an isometry $\phi: C\rightarrow [0,1]^n\subset\R^n$.  Any two charts differ by an isometry of $[0,1]^n$, which we identify with $O(n,\Z)$ as in the introduction. 

Now let $f:X\rightarrow X$ be an isometry. Suppose there exist two cubes $C, D\subset X$ with charts $\phi:C\rightarrow \R^n$ and $\psi:D\rightarrow \R^m$ such that $f(\text{int}(C))\cap\text{int}(D)\neq \emptyset$. Since $C,D$ are convex subsets of $X$, $f(\text{int}(C))\cap\text{int}(D)$ is convex if it is non-empty, hence it is connected.  

\begin{definition} Set $V=\text{int}(C)\cap f^{-1}(\text{int}(D))$. The \textit{trace} of $f$ at $C,D$ is defined to be the orthogonal affine map $\tau_{C,D}(f)=\psi\circ f \circ \phi^{-1}:\phi(V)\rightarrow\R^m.$ If $V$ is an open subset of $C$, $\tau_{C,D}(f)$ naturally extends to an orthogonal affine map $\R^n\rightarrow \R^m$, which we also denote by $\tau_{C,D}(f)$.
\end{definition}

Note that the trace is only well-defined up to a choice of chart, \emph{i.e.} up to pre- and post-composition with elements of the signed permutation group.  In what follows, we suppose we have $C,D$ as above with $n=m$, and that $V$ is an open subset of $C$.  Consider the standard cubulation of $\R^n$. We may write $\tau_{C,D}(f):\R^n_C\rightarrow \R^n_D$ as a map $x\mapsto Ax+b$, where $A\in O(n)$ and $b=(b_i)\in \R^n$ satisfies $\|b_i\|<\sqrt{2}/2$. In particular, if $b_i$ is an integer, then $b_i=0$. We call $A$ the \textit{orthogonal part} of $\tau_{C,D}$ and $b$ the \textit{translation part}. 

The standard cubulation of $\R^n$ decomposes into affine integer translates of lower dimensional cubulated Euclidean spaces, which we will call \textit{hypersurfaces}. 
\begin{lemma} \label{sec:Trace:lem:Orthodecomp} An isometry $T:x\mapsto Ax+b$ of $\R^n$ takes a proper hypersurface to a proper hypersurface if and only if up to conjugation by a permutation matrix, $A$ decomposes an orthogonal direct sum $A=A_1\bigoplus A_2$, where neither $A_i$ are empty, $b$ decomposes as an orthogonal sum $b=b_1\oplus b_2$ and the components of $b_2$ are integral.
\end{lemma}
\begin{proof} The second half is clearly sufficient, since if $T$ can be conjugated into this form, the subspace corresponding to the first $\text{rk}(A_1)$ coordinate vectors gets mapped to an integer translate of itself.  For the forward direction, if $b$ does not have any integer components, then clearly $T$ does not take hypersurfaces to hypersurfaces.  Hence consider some maximal collection of basis vectors $\{e_1,\ldots, e_k\}$ which span a coordinate $k$-plane $V$, and such that $b$ has integral coordinates $b_{k+1},\ldots, b_{n}$.  Each hypersurface is an integral affine translate of a subspace of this form, after permuting basis vectors.  After composing $T$ with an integer translation, we have that $T(0)=(x_1,\ldots,x_k,0\ldots, 0)$. It follows that $T(V)=V$ if and only if $A$ has the block form \[\left(\begin{array}{c|c}
A_1 & B\\ \hline
0 & A_2
\end{array}\right).\]
Hence we must have that $A_1\in O(k)$. We claim that $B$ must be 0.  Indeed, if $T'=T-T(0)$, then $T'$ is orthogonal linear and $T'(V)=V$.  Now post-compose $T'$ with multiplication by the matrix $A_1^t\bigoplus I_{n-k}$, where $I_{r}$ is the $(r\times r)$-identity matrix.  The resulting linear map $T''$ is represented by the matrix in block form \[\left(\begin{array}{c|c}
I_k & A_1^t\cdot B\\ \hline
0 & A_2
\end{array}\right).\] Since $T''$ is an orthogonal transformation, we must have that $A_1^t\cdot B=0$.  Hence, $B=0$ and $A_2\in O(n-k)$, as required. 
\end{proof}

\begin{definition}\label{ProperDecomp}
A Euclidean isometry $T:x\mapsto Ax+b$ \textit{preserves a proper hypersurface} if up to conjugation by a permutation matrix we can write $\R^n=\R^{n_1}\oplus \R^{n_2}$, $A=A_1\oplus A_2$ and $b=b_1\oplus b_2$ such that for $i=1,2$
\begin{itemize}
\item $n_i=\text{rk}(A_i)$,
\item $b_i\in \R^{n_i}$,
\item $b_2$ is integral. 
 \end{itemize} In this case, $T$ preserves the hypersurface $\R^{n_1}$.
 
\end{definition}

As in the situation above, we consider the trace of the map $\tau=\tau_{C,D}(f)$ at the corner $v_0\in C$. By changing the chart we can assume $\phi$ maps $v_0$ to $0\in \R^n$ and that $\tau_{C,D}$ has the form $x\mapsto Ax+b$ with $\psi\circ f(v_0)=\tau(0)=b\in [0,1]^n$. In terms of traces of isometries, the previous lemma characterizes when an isometry preserves some cubical subcomplexes, at least locally.  In other words, if $\tau_{C,D}$ preserves some hypersurfaces, by changing chart we can assume $b_2$ is 0 and the components of $b_1$ lie in $[0,1].$ When an isometry does not preserve any proper hypersurfaces, it is transverse in the following sense

\begin{lemma} \label{sec:Trace:lem:Transverse}Suppose an isometry $T:x\mapsto Ax+b$ does not preserve any proper hypersurfaces. Then each cube $K\subset \R^n$ of dimension $d$ with $1\leq d\leq n-1$ contains points which are mapped into the interior of higher-dimensional cubes.
\end{lemma}
\begin{proof} 

The vector $b$ may have some integer coordinates and some non-integer coordinates.  After a permutation, assume that $b_1,\ldots, b_k$ are all integers and $b_{k+1},\ldots, b_{n}$ are not. For any subcollection $e_{i_1},\ldots, e_{i_r}$, consider the $r$-cube $t_1\cdot e_{i_1}+\cdots+ t_r\cdot e_{i_r}$, with the $t_j\in [0,1]$.  We will show that this contains infinitely many points which lie in the interior of an $r+1$-cube.  Under $\tau$, the $j$th coordinate of a point $t_1\cdot e_{i_1}+\cdots+ t_r\cdot e_{i_r}$ is \[c_j=b_j+\sum_{k=1}^r t_k\cdot a_{ji_k}.\]
If $c_j$ is an integer on the entire $r$-cube then we must have $b_j$ is an integer and each $a_{ji_k}=0$. But if this occurs for at least $n-r$ different $j$, then $n-r\leq k$ and $A$ decomposes into an orthogonal direct sum $A_1\bigoplus A_2$ where $A_2$ corresponds to $e_{i_1},\ldots, e_{i_r}$ and $A_1$ is corresponds to a subset of the directions $e_1,\ldots, e_k$.  This contradicts the assumption about proper hypersurfaces by Lemma \ref{sec:Trace:lem:Orthodecomp}. It follows that for $1\leq r\leq n-1$, the image of each $r$-cube must have at least $r+1$ coordinates which are non-integral at infinitely many points, as desired.
\end{proof}


Consider now the general case of an isometry $T:x\mapsto Ax+b$ which preserves a proper hypersurface. After conjugation by a permutation matrix and an integer translation we may assume $T$ has the following form, $A=\Lambda\oplus B_0\oplus B_1  \cdots \oplus B_k$, and  $b=0\oplus b_0\oplus0\oplus \cdots \oplus 0$. Here, $\Lambda$ is a diagonal matrix with $\pm1$'s on the diagonal. In terms of Definition \ref{ProperDecomp}, $B_0$ corresponds to the preserved proper hypersurface, but $A_2$ has been decomposed further into blocks $\Lambda\oplus B_1\oplus\cdots\oplus B_k$. $\Lambda$ corresponds to the directions in which the derivative of $T$ is cubical.  For convenience, the first $l=\text{rk}(\Lambda)$ directions correspond to $\Lambda$. 

Let $L$ be the cube spanned by the first $l$ coordinate vectors.  We can post compose $T$ with $\Lambda\oplus I_{n-l}$ to make the derivative of $T$ the identity on $L$.  Note that $T(L)$ is parallel to $L$, and if $b=b_0=0$, then $T(L)= L$. Applying the previous lemma repeatedly we obtain
\begin{lemma}\label{GeneralTransverse}Suppose an isometry $T:x\mapsto Ax+b$ preserves a proper hypersurface.  Then 
each cube $K\subset \R^n$ of dimension $d$ with $1\leq d\leq n-1$ contains points which are mapped into the interior of higher-dimensional cubes, unless $K\subset L\times p$ where $p$ is an integer vector and $T(0\times p)$ is also an integer vector. 
\end{lemma}
\begin{proof}
The cubical structure on $\R^n$ is just the product of the cubical structures on the factors. Write $\R^n$ as $\R^l\oplus \R^{n_0} \oplus \R^{n_2}$ where $l=\text{rk}(\Lambda)$, $n_0=\text{rk}(B_0)$ and $n_2=n-l-n_0$ is the complementary dimension. Accordingly we can decompose $K=L'\times K'\times M'$. If $K'$ or $M'$ is not a vertex, then the conclusion above follows from Lemma \ref{sec:Trace:lem:Transverse}. 

%

 The remaining case is that $K=L'\times p$, where $p$ is an integral vector and $1\leq \dim L'\leq n-1$. $T$ restricts to the identity on $L'$ but may not be the identity on $p$. If $T(0\times p)$ is not an integral vector, then $T(K)$ meets some higher dimensional cubes. If $T(0\times p)=q$ is integral, then $T(K)=L'\times q$ is an integer translate of $K$.
\end{proof}
\begin{rmk} A slightly stronger statement fell out of the proof. If $p$ is an integral vector, $K\subset L\times p$ and $T(0\times p)$ is integral, then $T(K)$ is actually an integral translate of $K$.
\end{rmk}

\subsection{Finding flats in CAT(0) cube complexes.}\label{Flats} The goal now is to apply the lemmas of the previous section to the trace of an isometry $f:X\rightarrow X$ in order to find cubical subspaces which are locally Euclidean, except possibly at vertices.  Let $C$ be a locally maximal cube so that $f(C)$ is not cubical, and let $D$ be a locally maximal cube whose interior meets the interior of $f(C)$.  We let $\tau=\tau_{C,D}$ be the trace of $f$ at $C,D$.  After changing the chart, we assume $\tau:x\mapsto Ax+b$ is in the same form as in Lemma \ref{GeneralTransverse}. 

 As in the previous section, it is easier first to start with the case where no proper hypersurfaces are preserved, and then extend it to the general case.

\begin{proposition} \label{FindingFlats} Given $f$, $C,$ and $D$ as above, suppose that $\tau_{C,D}$ does not preserve any proper hypersurface. Then there exist convex subcomplexes $K_C \supset C$ and $K_D\supset D$ of $X$ such that 
\begin{enumerate}
\item Except at a set of vertices $\Lambda_C\subset K_C$ and $\Lambda_D\subset K_D$, the complexes $K_C$ and $K_D$ are locally Euclidean of dimension $n$.
\item All points of $K_C\setminus \Lambda_C$ and $K_D\setminus \Lambda_D$ have isomorphic links.
\end{enumerate}
\end{proposition}

\begin{rmk}We do not assume $C,D$ are locally maximal in this lemma. Thus, part (2) says that the link of any point in $K_C$ or $K_D$ is isomorphic to the join of $\Sigma_{n-1}$ and $\lk(C)=\lk(D)$. 
\end{rmk}

\begin{proof} Let $C$ and $D$ be $n$-cubes such that $f(\text{int}( C))$ intersects $\text{int} (D)$ in an open set, and let $\tau=\tau_{C,D}:\R^n\rightarrow \R^n$ be the associated trace map. We will show that locally, the map $f$ is determined by $\tau$, and implies the existence of subcomplexes $K_C$ and $K_D$ as claimed.

We will build the complexes $K_C$ and $K_D$ as follows. First set $K_C^0=\Int(C)$ and $K_D^0=\Int(D)$. By assumption, an open subset $U\subset \Int(C)$ maps isometrically onto an open subset of $\Int(D)$. It follows that all points in $K_C^0$ and $K_D^0$ have isomorphic links.  If $F$ is a face of $C$ which maps under $f$ to an interior point of $D$, then all points in $\Int(F)$ and $\Int(D)$ also have isomorphic links.  In particular, there exist a collection of $n$-cubes $E_1,\ldots, E_r$ containing $F$ whose interiors map into the interior of $D$, and thus make up a Euclidean neighborhood of $F$. Moreover, the isometry $f$ restricted to these cubes is just $\tau$, restricted to the corresponding cubes in $\R^n$ by continuity.  

With this inductive argument in mind, we proceed as follows.  Suppose we have already determined $K_C^i$ and $K_D^i$.  Let $F$ be a face in the closure $\overl{K_C^i}$ of $K_C^i$, and suppose that $f(F)$ meets the interior of some $n$-cube $E'$ in $K_D^i$.  Then at all points in $F$, the link is the same as in $E'$ and we can find cubes $E_1,\ldots, E_r$ in $\lk(F)$ which map to a locally $n$-dimensional euclidean neighborhood of $f(F)\cap\Int(E')$.  In particular, for each $j$ we have $f(\Int(E_j))\cap\Int(E')\neq \emptyset$. Set $N(F)=\Int(F)\cup\Int(E_1)\cup\cdots\cup\Int(E_r)$.  If $F_1,\ldots F_k$ are all the cubes in $\overl{K_C^i}$ whose image meets the interior of an $n$-cube of $K_D^i$, define \[K_C^{i+1}=K_C^i\cup\bigcup_{i=1}^kN(F_i),\] where we identify points in $K_C^{i+1}$ if they map to the same point under $f$. We define $K_{D}^{i+1}$ similarly, using $f^{-1}$.  

Observe that if $p\in K_C^i$ or $K_D^i$ and $q\in \Int(C)$, then $\lk(p)\cong \lk(q)$. To see this, as noted above it holds for all points in $K_C^0$ and $K_D^0$. For the inductive step, assume it holds for $K_C^{i}$ and $K_D^i$. If $p\in K_C^{i+1}$ or $K_D^{i+1}$, then $p$ lies in the interior of some cube $E$ which meets the interior of an $n$-cube of $K_C^i$ or $K_D^i$.  All points in $\Int(E)$ have the same link, and by induction all points in $K_C^i$ and $K_D^i$ have the same link as points in $\Int(C)$, which proves the assertion. 

We further remark that $K_C^i$ and $K_D^i$ are both connected open subsets of $X$, and that if $E$ is an $n$-cube such that $\Int(E)\subset K_C^i$ then there exists a chart on $E$ so that the orthogonal part of the trace of $f$ at $E$ is just $\tau_{C,D}$. Set $K_C^\infty=\cup_{i=1}^\infty K_C^i$ and define and $K_C$ to be the closure $K_C=\overl{K_C^{\infty}}$ (we define $K_D^\infty$ and $K_D$ similarly).  In order to prove the lemma, it suffices to show that $K_C$, $K_D$ satisfy both parts of the proposition. By the previous paragraph, we know that it holds for every point in $K_C^\infty$ and $K_D^\infty$.

\begin{lemma} \label{FaceMeet}Let $E$ be $n$-cube such that $\Int(E)\subset K_C^{i}$, and $E'$ an $n$-cube such that $\Int(E')\subset K_D^{i}$. 
\begin{enumerate}
\item If $f(\Int(E))\cap \Int(E')\neq \emptyset$, then there exists a face of $E$ which meets $\Int(E')$ and vice versa.   
\item There exists an integer $R=R(n)>0$ such that if $F$ is a codimension one face of $E$, then $\Int(F)$ is contained in $K_C^{i+R(n)}$.  
\end{enumerate}
\end{lemma}
\begin{proof}For (1), since $E$ and $E'$ are both $n$-cubes and $f$ is an isometry, the only way $f(\Int(E))\cap \Int(E')\neq \emptyset$ but that no face of $E$ meets interior of $E'$ is if $f(\Int(E))=\Int(E')$. But in this case, by continuity, $f$ must be cubical, contradicting our assumption on $f|_{\Int(E)}$. 

For (2), if $\Int(E)\subset K_C^{i}$, then $X$ is locally Euclidean on the image of $E$, and at $E$, the isometry $f$ is given by a trace $\tau$ which does not take proper hypersurfaces to proper hypersurfaces. Since $X$ is locally Euclidean at the image of $\Int(E)$, there must be a cubical Euclidean neighborhood of $f(\Int(E))$, say $D_1,\ldots, D_k$, such that $f(E)\subset \cup_iD_i$, and $\Int(D_i)\cap f(\Int(E))\neq \emptyset$ for each $i$. If $F\subset E$ is a codimension 1 face of $E$, then if $F$ meets the interior of an $n$-cube $D$, so does $E$. By Lemma \ref{sec:Trace:lem:Transverse}, we conclude that points of $F$ meet the interior of at least one of the $D_i$.  To prove the claim, it suffices to give a bound on $k$. By considering the distance to the central point of $E$, we see that the interior of $E$ meets at most $2^n$ cubes. Take $R(n)=2^n$.
%
%
%
\end{proof}
\begin{rmk} The constant $R(n)$ here is probably far from optimal, but for the purposes of the proposition we just need a bound to get that every codimension one face appears in $K_C^{i}$ eventually.
\end{rmk}

Following Lemma \ref{FaceMeet}, if the interior of some $n$-cube $E$ is contained in $K_C^{\infty}$, then so is the interior of each codimension one face of $E$. Moreover, restricted to each cube of $K_C^\infty$, the isometry $f$ looks like $\tau_{C,D}$, for a suitable choice of chart. Now for any cube $F$ in the closure of $K_C^\infty$ with $1\leq \dim(F)\leq n-2$, we know that by Lemma \ref{sec:Trace:lem:Transverse}, there exist points in $\Int(F)$ which map to higher dimensional cubes.  It follows that if $1\leq\dim(F)\leq n-2$, the link of points in $\Int(F)$ is the same as at all other points in $K_C^\infty$. 

The remaining cubes in the closure of $K_C$ are $0$-cubes, whose links may or may not be isomorphic to the link of each point in $K_C^\infty$. Denote the set of vertices which do not have the same link by $\Lambda_C$.  Similarly, $K_D$ is locally Euclidean except perhaps at a collection of vertices $\Lambda_D$. This proves that $K_C$ and $K_D$ satisfy claims (1) and (2) of the proposition.

Observe that $f$ takes $K_C$ isometrically onto $K_D$, mapping $\Lambda_C$ onto $\Lambda_D$. Thus, to prove convexity, it suffices to show that $K_C$ is a convex subcomplex of $X$.  Consider the inclusion map $\iota:K_C\rightarrow X$; we claim that $\iota$ is a local isometry. It is clearly injective, so we just need to show links of vertices embed as full subcomplexes.  

We know that at any point $p \in K_C\setminus \Lambda_C$, the link decomposes as $\lk(p)\cong \Sigma_n*L$ where $\Sigma_n$ is the standard simplicial $n$-sphere and $L\cong\lk(C)$. Now Let $v\in K_C^{0}$ be a vertex. Note that if $E$ is any cube containing $v$, then $\lk(E)\subset \lk(v)$ in a natural way: since $X$ is CAT(0), we have that $\lk(v)$ is a simplicial complex, and $E$ represents a simplex $\Delta$ in $\lk(v)$. Then $\lk(E)$ is isomorphic to $\lk(\Delta)$ as a subcomplex of $\lk(v)$.  Suppose $e_1,e_2\in \lk(v)\cap K_C$ are two vertices corresponding to edges which meet $v$. By the above observation, the link of $e_i$ in $\lk(v)$ is just $\Sigma_{n-1}*L$, where $L$ consists of cubes which are not in $K_C$. Note that any 1-cube in a join $A*B$ either connected two vertices in $A$, two vertices in $B$ or a vertex in $A$ and vertex in $B$.  Then if there is an edge $s\subset \lk(v)$ connecting $e_1$ and $e_2$, we must have that $s\subset \lk(v)\cap K_C$. Since $\lk(v)\cap K_C$ is clearly flag, this implies it is a full subcomplex of $\lk(v)$.  Hence $\iota$ is a local isometry.  
\end{proof}

\subsection{Local structure at singular points}\label{Link}
Let $K_C$ and $K_D$ be subcomplexes as in the proof of the proposition.  As we saw, except at a collection of vertices $\Lambda_C$, the points of $K_C$ all have isomorphic links.  Define $\lk(K_C)$ (resp. $\lk(K_D)$) to be the common link of all points in $K_C$ (resp. $K_D$). We will refer to the vertices in $\Lambda_C$ and $\Lambda_D$ as \textit{branch points} if $\lk(p)\setminus \lk(p)\cap K_C\neq \emptyset$. We call a branch point $p$ a \textit{singular} if $\lk(p)\cap K_C\not\cong\Sigma_{n-1}$. Let $p\in K_C$ be such a branch point and $q=f(p)$ its image in $K_D$.  Consider the link $\lk(p)$ of $p$ in $X$ and define $M_p=\lk(p)\cap K_C$.  The next proposition describes the structure of $M_p$.  

\begin{proposition} \label{LinkStructure}As above suppose that $p\in K_C$ is a branch point and that away from branch points, $K_C$ is $n$-dimensional with $n\geq 2$. If $n\neq2$, then $M_p\cong \Sa^{n-1}$, the $(n-1)$-sphere.  If $n=2$, then $M_p\cong \Sa^1(k)$ for some $k\geq 4.$ 
\end{proposition}
\begin{proof}
Because $X$ is locally finite,  we know that $M_p$ is compact. The star $\st(p)\cap K_C$ is a cone over $M_p$, and admits the structure of a simplicial complex $L$ so that $M_p$ is a subcomplex. We may subdivide so that $M_p$ lies in the interior of $L$ As $K_C$ is locally Euclidean away from $p$, we have $L$ is locally Euclidean away from $p$. By the simplicial neighborhood theorem \cite{RS72}, a regular neighborhood of $M_p$ in $L$ is a closed $n$-manifold with boundary two disjoint copies of $M_p$. Hence, $M_p$ is a compact $(n-1)$-manifold. Moreover, the CAT(0) cube complex structure on $X$ induces a metric simplicial structure on $M_p$ so that $M_p$ is a simplicial manifold built out of regular spherical $(n-1)$-simplices. The isometry $f$ maps $M_p$ isometrically onto $M_q$, which similarly has a spherical simplicial structure.  

The fact that $f$ does not preserve any proper hypersurfaces implies that each simplex in $M_p$ has a locally spherical metric, \emph{i.e.} has a neighborhood isometric to the standard round sphere $\Sa^{n-1}$. More precisely, let $C$ be a cube incident at $p$ with $1\leq \dim(C)\leq n-1$. Then $C\cap\lk_\epsilon(p)$ is a simplex $\sigma$ of dimension $\dim(C)-1$. By Lemma \ref{sec:Trace:lem:Transverse}, for infinitely many $\epsilon$, there exist points in the interior of $\sigma$ which map into the interior of a simplex in $\lk_\epsilon(f(p))$ of dimension 1 higher.  Since $\lk(p)$ has dimension $n-1$, this holds for all $\sigma$ in $\lk(p)$, every point in $\lk(p)$ must have a neighborhood isometric to the interior of an $n-1$-simplex.  The latter is isometric to $\Sa^{n-1}$.

To complete the proof of the proposition, it suffices to prove the next lemma, which implies that for $n\geq 3$, for any singular point $p$, there is only one possibility for $M_p$.


\begin{lemma} Suppose $m\geq2$ and let $M$ be a spherical $m$-manifold which has a simplicial structure made up of regular spherical $m$-simplices. Then $M$ is isometric to $\Sa^m$.  
\end{lemma}
\begin{proof} Let $\Delta$ be the regular spherical $m$-simplex. The fact that $M$ is built out of regular spherical $m$-simplices means that the combinatorics of the simplicial structure on $M$ must be locally the same as that of the standard tiling of $\Sa^m$ by regular $m$-simplices. For $0\leq i\leq m$, let $F_k$ be the number of spherical $k$-simplices.  The standard simplicial structure on $\Sa^m$ gives equations for the number of $m$-simplices incident on a $k$-simplex:\[2^{m-k}\cdot F_k= {{m+1}\choose{k+1}}\cdot F_m.\] Since $M$ is a simplicial complex, each edge is determined by its endpoints.  In particular, by the pigeonhole principle we must have that \[{{F_0}\choose{2}}\geq F_1.\] Substituting the equations for $F_0,F_1$ in terms of $F_m$ we obtain:\[\frac{(m+1)^2}{2^{2m+1}}F_m^2-\frac{(m+1)}{2^{m+1}}F_m\geq\frac{m(m+1)}{2^m}F_m.\] Dividing by $F_m>0$, clearing denominators and rearranging we get:\begin{align*}(m+1)^2F_m&\geq 2^m(m+1)+2^{m+1}(m^2+m)\\
&=2^m(2m^2+3m+1)\end{align*}
From this we conclude that $F_m\geq 2^{m+1}$.  If $M$ has $N\geq2^{m+1}$ distinct $m$-simplices and is a degree $d\geq1$ quotient of $\Sa^m$, then by passing to the universal cover, we get a simplicial complex structure on $\Sa^m$ consisting of $Nd$ regular $m$-simplices.  Hence we must have \begin{align*}\text{vol}(\Sa^m)&=Nd\cdot\text{vol}(\Delta)\\
&\geq 2^{m+1}d\cdot \text{vol}(\Delta).
\end{align*}
On the other hand, a simple computation shows that $\text{vol}(\Sa^m)=2^{m+1}\cdot\text{vol}(\Delta)$, hence we conclude that $d=1$ and $M\cong \Sa^m$.  
\end{proof}
By the lemma, when $n\geq 3$, then $n-1\geq 2$ and $M_p\cong \Sa^{n-1}$.  When $n=2$, we still have that $M_p$ must be homeomorphic to $\Sa^1$, since this is the only compact 1-manifold. In fact $M_p$ will be isometric to $\Sa^1(k)$, the circle of total angle $k\cdot \pi/2$.  Since $M$ is a simplicial complex and $X$ is NPC, $k\geq 4$. Thus, in dimension 1, there are countably many possibilities, depending only on the total angle. 
\end{proof}

\subsection{The general product structure}
Combining the results of sections \ref{Flats} and \ref{Link}, up to this point we have proven the following
\begin{proposition} \label{SpecialRn} Given $f$, $C,$ and $D$ as above, suppose that $\tau_{C,D}$ does not preserve any proper hypersurface. Then there exist convex subcomplexes $K_C \supset C$ and $K_D\supset D$ of $X$ such that 
\begin{enumerate}
\item $K_C$ and $K_D$ are homeomorphic to $\R^n$ and except at a set of vertices $\Lambda_C\subset K_C$ and $\Lambda_D\subset K_D$, the complexes $K_C$ and $K_D$ are locally Euclidean of dimension $n$.
\item All points of $K_C\setminus \Lambda_C$ and $K_D\setminus \Lambda_D$ have isomorphic links.
\end{enumerate}
\end{proposition}

We can now state the most general form in the case where $C$ and $D$ are locally maximal and $\tau=\tau_{C,D}$ may preserve a proper hypersurface.  If $H$ is a preserved hypersurface, we denote by $\Lambda_H$ the collection of branch points given by the previous proposition.

\begin{thm} \label{GeneralFindingFlats} Given $f$, $C,$ and $D$ as above, there exist convex subcomplexes $K_C \supset C$ and $K_D\supset D$ of $X$ such that 
\begin{enumerate}
\item $K_C=[0,1]^l\times K_C'$, $K_D=[0,1]^l\times K_D'$. Both $K_C'$, $K_D'$ are  homeomorphic to $\R^{n-l}$ and locally Euclidean except at a collection of proper hypersurfaces and vertices.

\item If $H$ is a proper hypersurface preserved by $\tau$,  then all points of $H\setminus \Lambda_H$ have isomorphic links.
\end{enumerate}
\end{thm}
\begin{proof}
The same proof as in Proposition \ref{FindingFlats} goes through by using Lemma \ref{GeneralTransverse} in place of Lemma \ref{sec:Trace:lem:Transverse}. In this proof, we will not explicitly mention $K_D$, but its existence will follow from the existence of $K_C$ together with $f$ and our repeated application of Proposition \ref{SpecialRn}. As in the set up to Lemma \ref{GeneralTransverse} we write $\tau:x\mapsto Ax+b$ where $A=I_l\oplus A_1\oplus B_1\oplus \cdots \oplus B_k$ and $b=b_1$. We can also decompose $C$ as $L\times C_0\times \cdots \times C_k$, where $L$ corresponds to $I_l$, $A_1$ corresponds to $C_0$ and $C_i$ corresponds to $B_i$ for $1\leq i\leq k$. Let $n_i$ be the dimension of $C_i$  for $0\leq i\leq k$. Define maps $\tau_i:\R^{n_i}\rightarrow \R^{n_i}$ by $\tau_0:x\mapsto A_0x+b_0$ and $\tau_i:x\mapsto A_ix$ for $1\leq i\leq k$.

We claim that there exists a full subcomplex $K_C$ containing $C$ which decomposes as a product $[0,1]^l\times\R^{n_0}\times\cdots\times \R^{n_k}$.  Moreover, we claim that along an open set which is the complement of sets of the form $L' \times \Lambda_{C_0}\times \cdots \times \R^{n_i}\times \cdots\times \Lambda_{C_k}$ with $L'\subset L$, the link of all points is the same. Here $\Lambda_{C_i}$ is a collection of vertices in $\R^{n_i}$.

The idea is to apply Proposition \ref{SpecialRn} repeatedly.  In the first step, we apply Proposition \ref{SpecialRn} to $C_0$ to obtain a full subcomplex $K_{C_0}$. The second part of Proposition \ref{SpecialRn} implies that away from a set of vertices $\Lambda_{C_0}$, the link of all points is the same.  Hence each point in $K_{C_0}\setminus \Lambda_{C_0}$ has link isomorphic to $\Sigma_{n_0-1}*\lk(C_0)$. Moreover, every point in $K_{C_0}$ contains a flag subcomplex in its link corresponding to $L$, and we obtain a full subcomplex $L\times K_{C_0}$.

Consider a point $p\in \Int(C_0)$. There is a cube isomorphic to $C_1$ passing through $p$.  Choose a chart $\phi$ for $C$ and $\psi$ for a cube of dimension $n$ containing $f(p)$, and which contains a neighborhood of $p$ in $\{p\}\times C_1$. Locally $\tau$ preserves the decomposition $C_0\times C_1$, so that $\phi(p)\times\R^{n_1}\subset\R^n$ maps to $\psi(f(p))\times\R^{n_1}\subset \R^n$. We can project $\pi_1:\R^n\rightarrow \R^{n_1}$ so that the composition \[\R^{n_1}\hookrightarrow \phi(p)\times \R^{n_1}\rightarrow \R^n\twoheadrightarrow \R^{n_1},\] is just the map $\tau_1$ up to elements of $O(n_1,\Z)$. Now we apply Lemma \ref{sec:Trace:lem:Transverse} to this composition, and the same proof as in Proposition \ref{FindingFlats} shows that we are able to find a subcomplex $K_{C_k}$ containing $p\times C_k$, which is locally Euclidean except potentially where vertices map to vertices under $\tau_1$.  Denote this set of vertices by $\Lambda_{C_1}$.  Since the link at all points in $K_{C_0}\setminus\Lambda_{C_0}$ is the same, we find subcomplexes through every point of $K_{C_0}$ isometric to $K_{C_1}$.  At points in $\Lambda_{C_0}$, the link contains a flag subcomplex corresponding to $\lk(C_0)$, so we are also able to repeat the argument here. Note that the link at all points of $K_{C_0}\times K_{C_1}\setminus (\Lambda_{C_0}\times K_{C_1}\cup K_{C_0}\times \Lambda_{C_1})$ is the same, and is isomorphic to $\Sigma_{n_0}*\Sigma_{n_1}*\lk(C_0\times C_1)$. In particular, we are again able to extend the product to $L\times K_{C_0}\times K_{C_1}$.

Repeating this argument, we find a full subcomplex $K_C=[0,1]^l\times K_{C_0}\times \cdots \times K_{C_k}$, as well as branch vertices $\Lambda_{C_i}\subset K_{C_i}^{(0)}$.  Away from the collection of hypersurfaces $ \Lambda_{C_0}\times\cdots\times K_{C_i}\times\cdots \times\Lambda_{C_k}$, the link of each point intersected with $K_C$ is the same and equals $\Sigma_{n_0-1}\times \cdots \times\Sigma_{n_k-1}$. At other points in $K_{C_0}\times \cdots \times K_{C_k}$, the link of a point is still a join of simplicial spheres by Proposition \ref{LinkStructure}, hence it is still a simplicial sphere.  $K_C$ is a connected full subcomplex, hence CAT(0) and contractible. Since the link of each point of $K_{C_0}\times \cdots \times K_{C_k}$ is a simplicial sphere, $K_C=[0,1]^l\times K_{C_0}\times \cdots \times K_{C_k}\cong [0,1]^l\times \R^n$.


Branching may occur along cubes in $L\times \Lambda_{C_0}\times\cdots \times K_{C_i}\times \cdots\times \Lambda_{C_k}$.  Let $H$ be a preserved hypersurface $H=v\times \lambda_0\times K_{C_i}\times \cdots\times \Lambda_{C_k}$, where $v\in L^{(0)}$ is a vertex, and $\lambda_j\in \Lambda_{C_j}$. The second part of Proposition \ref{SpecialRn}, ensures that all points in $H\setminus \Lambda_{C_i}$ have isomorphic links.
\end{proof}
As we saw above, the proof in fact determines the structure of $K_C$, $K_D$ more precisely as a product.  When combined with Proposition \ref{LinkStructure} we obtain
\begin{corollary}\label{ProductDecomp}With the same notation as in the previous theorem, let $C=L\times C_0\times C_1\times \cdots\times C_k$.
Then $K_C=[0,1]^l\times K_{C_0}\times K_{C_1}\times \cdots \times K_{C_k}\cong [0,1]^l\times \R^{n_0}\times\cdots \R^{n_k}$, where $n_i=\dim(C_i)$. Moreover, $\R^{n_i}$ is standard if $n_i\neq 2$, and possibly singular if $n_i=2$. 
\end{corollary}

\section{The isometry group of a CAT(0) cube complex}

\noindent
In this section we apply the results of the previous section to prove Theorem \ref{main1}. Given a non-cubical isometry $f$, the key idea is to find a subcomplex $K_C$ as above, and observe that we can extend $K_C$ to contain the entire hyperplane dual $C_0\times\cdots\times C_k$. We state the main theorem here in a slightly more technical form, which better describes the structure of $X$ locally near the $\R^n$ subcomplex.

\begin{thm} \label{FlatDecomp} Let $X$ be a CAT(0) cube complex, and suppose $f:X\rightarrow X$ is an isometry which does not take cubes to cubes.  Then there exists a full subcomplex $Z\subset X$ satisfying the following properties:
\begin{enumerate}
\item $Z$ decomposes as $Z=Y\times \R^{n_0}\times\cdots\times \R^{n_k}$ for some $n_i>0$ and some full subcomplex $Y$ of $X$.  
\item If $n_i\neq 2$, the $\R^{n_i}$ factor is standard, while if $n_i=2$, $\R^{n_i}$ may have a singular cone metric.  
\item If $n_0=1$, then $X=X_0\times\R$ for some subcomplex $X_0$.  If $n_0>1$, $Z$ may be branched along a subcomplex $Y\times\Lambda$, where $\Lambda$ is a subset of vertices and proper hypersurfaces.    
\item $Z$ is the closure of a connected component of $X\setminus (Y\times \Lambda)$. 
\end{enumerate}
\end{thm}
%
%
%
%
\begin{proof}
Suppose $f:X\rightarrow X$ is not cubical.  Choose $C\subset X$ a locally maximal cube of dimension $n$ such that $f(C)$ is not a cube.  Such a cube always exists since every cube is contained in at least one locally maximal cube.  Let $\tau:x\mapsto Ax+b$ be the trace of $f$ for some choice of chart at $C$.  As in the previous section, decompose $C$ as $C=L\times C_0\times \cdots C_k$ and $\tau$ as $A=I_l\oplus A_0\oplus\cdots\oplus A_k$ with translation part $b=0\oplus b_0\oplus\cdots\oplus 0$.  By Theorem \ref{GeneralFindingFlats}, we can find a subcomplex $K_C=[0,1]^l\times K_{C'}\cong [0,1]^l\times \R^{n_0}\times\cdots\times \R^{n_k}$. 

The cube $[0,1]^l$ corresponds to the directions in which $f$ is cubical. Consider the hyperplane dual to $C'=C_0\times \cdots C_k$, which meets $C$ in a cube parallel to $L$. Let $L'\subset L$ be a subcube. Along each subcomplex of the form $L'\times K_{C'}$ all points in dense open set have the same link, and moreover, $f$ is given by developing $\tau$, hence is still cubical in the directions given by $L'$. 

If $E$ is a maximal cube containing $L'\times C'$, then $E$ extends over all of $L'\times K_{C'}$.  In a chart for $E$, $f$ preserves the hypersurface coming from $ L'\times C'$.  If we write $E=E'\times L'\times C'$ then $f|_E$ preserves this decomposition.  Therefore, for each subcube of the form $\{p\}\times L'\times C'$ where $p\in E'^{(0)}$, $f|_E'$ is just given by $\tau$. This implies that there is a subcomplex $\{p\}\times L'\times K_{C'}$ passing through $\{p\}\times L'\times C'$. Note that if $f|_E$ is not cubical along $E'$, it may be that $E$ itself is contained in a larger subcomplex which splits off more Euclidean factors. 

We can extend the product decomposition to $E$ and keep going, since all points in a dense open subset of $\{p\}\times L'\times K_{C'}$ have the same link by Theorem \ref{GeneralFindingFlats}. In this way, the entire hyperplane $Y$ dual to $C'$ is contained in a subcomplex $Z$ which decomposes as $Z=Y\times K_{C'}$.  Observe that $Y$ is a convex subcomplex since it is an intersection of hyperplanes.  The fact that $Z=Y\times K_C'$ is a full subcomplex now follows from the fact that $K_{C'}$ is. This proves part (1) of the theorem and (2) follows from Corollary \ref{ProductDecomp}.

In the proof of Theorem \ref{GeneralFindingFlats} we found a collection of hypersurfaces along which branching may occur.  For simplicity of notation we have denoted the union of these sets by $\Lambda$.  If $C_0$ is 1-dimensional, and $f$ is not cubical at $C'$, then $b=b_0\neq 0$, and hence $\Lambda=\emptyset$. Otherwise, $n\geq2$ and the branch locus is codimension at least 2 in $Z$. This proves the second half of $(3)$.  Hence, $Z\setminus Y\times \Lambda$ is connected, and since $Z\setminus Y\times \Lambda$ is open, $Z$ is the closure of a connected component of $X\setminus Y\times \Lambda$, proving $(4)$. Finally, for the first half of (3), since $\Lambda=\emptyset$, $Z$ is a connected component of $X$.  As $X$ is connected, $X=Z=Y\times \R$.
\end{proof}

\subsection{Applications to rigidity}
To end the section, we deduce theorems \ref{main2} and \ref{main3} as corollaries of Theorem \ref{FlatDecomp}.  First, we prove Theorem \ref{main2}, concerning $\delta$-hyperbolic CAT(0) cube complexes. 
\begin{thm} Suppose $X$ is $\delta$-hyperbolic, cocompact and 1-ended.  Then $\Isom(X)=\Aut(X)$ unless $X$ is quasi-isometric to $\Hy^2$.  

\end{thm}
\begin{proof}Suppose $X$ is hyperbolic and that there exists an isometry $f:X\rightarrow X$ which does not take cubes to cubes. By Theorem \ref{FlatDecomp}, we find a subcomplex $Z=Y\times \R^n$ for some $n\geq 1$.  Since $X$ is hyperbolic, it cannot contain any flats of dimension $\geq 2$. It follows that $Y$ must be compact, and that $n=$ 1 or $n=2$, in which case $\R^2$ has a singular Euclidean metric.  If $n=1$, then by part (3) of the theorem, $X=Y\times \R$, contradicting our hypothesis that $X$ be 1-ended.  

Thus, we may assume $n=2$, and that $\R^2$ carries a singular Euclidean metric. If no branching occurs then $X=Y\times \R^2$. By Proposition \ref{ClassifySingularCone} and the cocompactness assumption, we conclude that $X$ is quasi-isometric to $\Hy^2$. Suppose therefore that branching occurs along a cobounded subset of vertices $\Lambda \subset \R^2$.  Since $Y$ is compact, we can consider a compact set $K$ containing a neighborhood of $Y\times \{\lambda\}$ for some $\lambda \in \Lambda$.  The complement of $K$ in $\widetilde{X}$ has at least one unbounded component, coming from $Z$.  If there is more than one unbounded component, then $X$ must have infinitely many ends, contradicting our assumption.  Therefore, $X$ is contained in bounded neighborhood of $Z$, and the action of $G$ preserves the $\R^2$ factor with the singular Euclidean metric.  As above, we conclude that $X$ is quasi-isometric to $\Hy^2$.  
\end{proof}
\begin{rmk} In fact, since by Theorem \ref{FlatDecomp}(1), $Y$ is a full-subcomplex, it is also CAT(0), and hence contractible. Since each component of $X\setminus(Y\times \Lambda)$ is bounded it is possible to collapse $Y$ and each component of $Y\times \Lambda$ \cite{Br16} so that $X\simeq Y\times\R^2_{\text{sing}}$.  If we allow $X$ to have more than one end, we get that $G$ is either virtually free or decomposes as a free product with a surface group.  
\end{rmk}

When $X$ is a singular cone metric on $\R^2$, it is still the case that the automorphism group is a lattice in the full isometry group, as the next proposition shows.
\begin{lemma}\label{FiniteIndexSing} Let $X$ be a CAT(0) cube complex which is a singular cone metric on $\R^2$. 
\begin{enumerate}
\item If $X$ has a single cone point, then $\Isom(X)\cong O(2)$. 
\item If $X$ has at least 2 cone points, then $\Aut(X)$ has finite index in $\Isom(X)$. 
\end{enumerate}
\end{lemma}
\begin{proof}
If $X$ has a single cone point $x_0$ of order $d$, then any isometry must preserve this cone point and the structure of $X$ is locally Euclidean everywhere else. It follows that any reflection in a geodesic through $x_0$ with equal angles on each side induces an isometry, and therefore any ray emanating from the cone point defines a unique isometry.  Moreover, we can rotate about $x_0$ through any angle $\theta\in [0,\frac{\pi d}{2})$. Conversely, if $f$ is any isometry such that fixes some point $p\neq x_0$, then $f$ fixes the ray from $x_0$ to $p$. Hence $f$ is either the identity or the reflection corresponding to that ray. If $f$ does not fix any points other than $x_0$, we can rotate so that it does.  This concludes the proof of $(1)$.

If $X$ has exactly 2 cone points, then $\Isom(X)$ fixes these two cone points and the geodesic connecting them setwise.  It follows that $\Isom(X)$ is a subgroup of the Klein 4 group $\Z/2\oplus \Z/2$. Finally suppose $X$ has at least three cone points. We will show that there exists $N>0$ such that for any isometry $f\in \Isom(X)$, $f^N$ is cubical. Let $x_0, x_1,$ and $x_2$ be three cone points which span a geodesic triangle $T$. Recall this is equivalent to the fact that there are no singular vertices in $\Int(T)$. Thus the convex hull of $x_0,x_1,x_2$ determines a Euclidean triangle $T'\subset \R^2$. By identifying $x_0$ with $(0,0)\in \R^2$, we see that it is possible to represent $x_1$, and $x_2$ as points in $\R^2$ with integer coordinates. The three distances $d_X(x_0,x_1)$, $d_X(x_0,x_2)$ and $d_X(x_1,x_2)$ therefore determine only finitely many possible configurations. Let $M$ be the number of these configurations.  

Consider an isometry $f\in \Isom(X)$ restricted to $T$. By the pigeonhole principle, one of the powers $f, f^2,\ldots, f^N$ must return the configuration of cubes around $T$ back to itself.  Taking $N=M!$, we see that for any isometry, $f^N$ is cubical on $T$. But then $f^N$ is cubical everywhere.    
\end{proof}

We will see in the next section that there may exist cocompact lattices of CAT(0) cube complexes which are not entirely comprised of cubical isometries. Moreover, there exist cocompact lattices of Euclidean space which do not contain any non-trivial automorphisms of the standard cubical structure. The next theorem says that unless $X$ has a Euclidean factor, every cocompact lattice will have a finite index subgroup consisting only of automorphisms.  

\begin{thm}Let $X$ be a CAT(0) cube complex, and $\Gamma\leq \Isom(X)$ a proper and cocompact subgroup.  Then either there exists a finite-index subgroup $\Gamma'\leq \Gamma$ such that $\Gamma'\leq \Aut(X)$, or $X\cong\E^n\times Y$ for some subcomplex $Y$.
\end{thm}
\begin{proof} The proof is by induction on the dimension of $X$.  If $\dim(X)=1$, then $X$ is a tree and hence either $X$ has a branch point in which case $\Isom(X)=\Aut(X)$, or $X=\R=\E^1$.  This proves the base case.  

Now assume $\dim(X)>1$ and suppose $\Gamma$ contains a least one non-trivial isometry $f$ that is not cubical. Let $C$ be a maximal cube at which $f$ is not cubical. By Theorem \ref{FlatDecomp}(3), if $f$ preserves a 1-dimensional hyperplane, then $X=Y\times \R$ and we are done. We therefore assume that $f$ does not preserve a 1-dimensional hyperplane, and find a convex subcomplex $Z\cong Y\times K_C$, where $K_C=\R^{n_0}\times \cdots \times\R^{n_k}$.

Note that if the restriction of $\Gamma$ to $Y$ is not cubical, by \ref{FindingFlats}(4) we would be able to decompose $Z$ further. Since $X$, and hence $Z$, is finite-dimensional, we can assume that the decomposition above is maximal and that $\Gamma$ always takes cubes in $Y$ to cubes. If $\Lambda=\emptyset$, then $Z=X$. By part (2) of \ref{FlatDecomp}, if some $n_i\neq 2$, then $\R^{n_i}=\E^{n_i}$ and we have shown $X$ has a Euclidean factor. Otherwise $n_i=2$ for every $i$ and every factor is $\R^2$ with a singular cone metric. Passing to a finite index subgroup, we can assume $\Gamma$ preserves some $\R^2$-factor. Now apply Lemma \ref{FiniteIndexSing}, and by induction, the theorem follows.

Otherwise, $Z$ is branched along $Y\times\Lambda$, where $\Lambda$ is a union of hypersurfaces of the form $\{v_0\}\times\cdots\times\R^{n_i}\times \cdots\times\{v_k\}$ and vertices, where the $v_j\in K_C^{(0)}$. Consider now the orbit $\Gamma.(Y\times\Lambda)$ and the complement $X'=X\setminus \Gamma.(Y\times\Lambda)$.  $\Gamma$ must permute the connected components of $X'$. Let $W$ be the closure of a connected component of $X'$. If $H_0=Y\times \{v_0\}\times\cdots\times\R^{n_i}\times \cdots\times\{v_k\}$ and $H_1=g.H_0\subset W$ is some translate of $H_0$, observe that since $f|_{H_0}$ is not cubical, either $g$ is not cubical or $fg$ is not cubical.  In each case, $H_1$ is a preserved hyperplane, and by applying Theorem \ref{FlatDecomp} to $W$, we see that $W$ decomposes as a product.  If some component $W$ is branched only along a subcomplex of $Y$, then we look at the restriction of $\Gamma$ to this component.  Either the restriction $\Gamma|_W$ is cubical, or $W$ decomposes as in Theorem \ref{FlatDecomp}. 

If the restriction of $\Gamma$ to $W$ is not cubical, we may need to make the branch set larger.  However, by the cocompactness of $\Gamma$, this process can happen only finitely many times.  Indeed take a large compact subset $K\subset X$ which surjects onto $X/\Gamma$, then $K$ meets a translate of each component of the branch set.  In the end, we obtain a $\Gamma$-invariant branch set $\Omega$, such that on each component $W$ of $X''=X\setminus \Omega$, the restriction of $\Gamma$ is either entirely cubical, or it isn't, in which case $W$ decomposes as in \ref{FlatDecomp}.

There are now two cases to consider.  In the first case, suppose branching only occurs along translates of some hypersurface $H_0=\{v_0\}\times\cdots\times\R^{n_i}\times \cdots\times\{v_k\}$, so that any two translates are either disjoint or are equal.  Thus, every connected component of $X''$ is non-cubical, and decomposes as a product with $H_0$.  Then $X=X_0\times H_0=X_0\times Y\times \R^{n_i}$. If $n_i\neq 2$, we are done, otherwise the desired conclusion follows by induction and Lemma \ref{FiniteIndexSing}.

For the second case, branching occurs along isolated vertices or on multiple hypersurfaces.  If the latter occurs, however, we can reduce to the case of isolated vertices by observing that any two hypersurfaces in $\R^{m_0}\times\cdots \times \R^{m_l}$ intersect in a unique point.  By considering these isolated branch points, we see that there are infinitely many in each component of $X''$. Choose representatives $X_1,\cdots, X_s, Z_1,\cdots, Z_r$ for the orbits of the connected components of $X''$, where the $\Gamma|_{X_i}$ is cubical and $Z_j=Y_i\times K_j$, and $K_j$ is a product of (possibly singular) Euclidean spaces. If $\Gamma|_{X_i}$ is cubical, then this does not change after we pass to finite index subgroups of $\Gamma$.  Hence, it suffices to show that there exists $N>0$ such that for any $f\in \Gamma$, $f^N|_{Z_j}$ is cubical. This fact will follow from the next lemma.

\begin{lemma} Consider a collection of points $\{v_0,\cdots, v_M\}\subset\E^n$, where $M\geq n$.  If the $\{v_i\}$ are affine independent, there are only finitely many configurations such that $\{v_1,\ldots,v_M\}$ are vertices of the standard cube complex structure on $\E^n$.
\end{lemma}
\begin{proof} Without loss of generality, we identify $v_0$ with the origin in $\E^n$. The other points $v_1,\ldots,v_M$ are vertices for the standard cube complex structure on $\E^n$, hence have integral coordinates $v_i=(v_{i1},\ldots,v_{in})$. Let $D$ be the maximum distance from $v_0$ to any of the $v_i$. Then there are only finitely many integral points in the ball of radius $D$ about the origin. Hence only finitely many configurations.  Since $M\geq n$, once the positions of the $\{v_i\}$ have been fixed their convex hull contains a simplex in $\E^n$, and the hence the cube complex structure on the rest of $\R^n$ is determined.  
\end{proof}

To finish proof, consider one of the $Z_i$, and suppose $Z_i$ decomposes as $Y_i\times \R^{m_0}\times\cdots \times \R^{m_l}$. Given any $f\in \Gamma$, we can regard $f|_{Z_i}$ as an isometry from $Z_i$ to itself, which must preserve the isolated branch points.  We claim that there exists $N_i$ such that $f^{N_i}$ is a cubical automorphism. By Lemma \ref{FiniteIndexSing}, we may assume that all of the $\R^{m_i}-$factors above are standard. Choose $K_i\subset Z_i$ large enough that it contains an affine independent subset of isolated branch points. This is possible since $Z_i$ is non-compact and contains an open set, hence its stabilizer in $\Gamma$ must act cocompactly.  By the previous lemma, if $N_i$ is the number of possible configurations of branch vertices in $K_i$, we see that as in the proof of lemma \ref{FiniteIndexSing}, $f^{N_i!}|_{Z_i}$ has to preserve the cubical structure, no matter what it is. Setting $N=N_1!\cdots N_r!$, we see that for any $f\in \Gamma$, $f^N$ is a cubical automorphism when restricted to the $X_i$ and all of the $Z_j$, hence is cubical everywhere.  
\end{proof}
\section{Non-cubical Cocompact Examples}

In this final section, we find examples of NPC cube complexes possessing non-cubical isometries.  The examples are all manifolds which are quotients of $\R^n$, which will be standard when $n\neq 2$ and possibly singular when $n=2$. We first construct the Euclidean examples.  Note that any translation of $\E^n$ descends to a non-cubical isometry of a torus quotient.  However, any translation is isotopic to the identity.  The examples below will act non-trivially on the fundamental group, hence will not be isotopic to the identity.

\begin{example}Let $\mathbf{a}=(a_1,a_2)$ and $\mathbf{b}=(b_1,b_2)$ be integral vectors in $\R^2$, and suppose that $\mathbf{a}, \mathbf{b}$ are distinct in the sense that their coordinates do not differ by a signed permutation. In particular, $\mathbf{a},\mathbf{b}$ are linearly independent. Suppose moreover that $a_1^2+a_2^2=b_1^2+b_2^2$. We call such pairs distinct Pythagorean doubles.  Now consider the standard cubical lattice on $\R^2$ and the torus quotient $\T^2_{\mathbf{a},\mathbf{b}}$ by $\Z^2\cong\langle \mathbf{a}, \mathbf{b}\rangle$. Since $\mathbf{a}, \mathbf{b}$ have the same length, there is an isometry which exchanges them.  This is an isometry which does not take cubes to cubes since it does not lie in $O(2,\Z)$, but it does preserve the lattice $\langle \mathbf{a}, \mathbf{b}\rangle$ and hence descends to an isometry of the quotient $\T^2_{\mathbf{a},\mathbf{b}}$ which is not cubical. For example take $\mathbf{a}=(1,8)$ and  $\mathbf{b}=(7,4)$.
\end{example}

The previous example is eaily extended to higher-dimensional tori either by taking products, or by finding distinct Pythagorean $n$-tuples in the same manner as above.  In dimension $n$, one can always obtain the symmetric group from a collection of distinct Pythagorean $n$-tuples which all have the same length. Thus we see that the index of the automorphism group is in the isometry group finite but may be arbitrarily large. Finally, we construct non-cubical isometries on higher-genus surfaces, whose universal covers will be singular Euclidean metrics on $\R^2$.

\begin{example}To get examples of higher genus surfaces in dimension 2, we take branched covers of the torus examples above.  Let  $\T^2_{\mathbf{a},\mathbf{b}}$ be a torus as above with non-cubical isometry $f$. The isometry $f$ fixes a circle on the torus which passes through a vertex $p$ which is the image of the lattice $\Z^2$.  Deleting the base vertex, we obtain a punctured torus whose fundamental group is the free group $F_2$.  Take a finite cover $X'$ of corresponding to a finite-index characteristic subgroup of $F_2$ which does not contain the commutator of the two generators.  Then by filling in the punctures in the cover we obtain a branched cover $\pi:X\rightarrow\T^2_{\mathbf{a},\mathbf{b}}$. The cubical structure on $\T^2_{\mathbf{a},\mathbf{b}}$ lifts to $X$ so that $\pi$ is a map of cube complexes.  Moreover, since the cover is characteristic, $f$ lifts to an isometry $\widetilde{f}:X\rightarrow X$ such that $\pi\circ \widetilde{f}=f\circ \pi$. As $\pi$ is cubical, we conclude that $\widetilde{f}$ is not cubical.
\end{example}

In light of the previous example, we see that both Theorem \ref{main1} and \ref{main2} are best possible, in the sense that singular Euclidean metrics on $\R^2$ may have non-cubical isometries.

%
%
%
%
%

\bibliography{IsomBib}
\bibliographystyle{plain}

\end{document}